\documentclass[12pt,a4paper,twoside,british,openright]{article}
\usepackage[T1]{fontenc}
\usepackage[latin1]{inputenc}
\usepackage{fancyhdr}
\usepackage{color}
\usepackage{calc}
\usepackage{mathtools}
\usepackage{amsmath}
\usepackage{amsthm}
\usepackage{amssymb}
\usepackage{mathdots}
\usepackage{undertilde}
\usepackage{setspace}
\usepackage{esint}
\makeatletter

\numberwithin{equation}{section}
\numberwithin{table}{section}
\numberwithin{figure}{section}
\theoremstyle{plain}
\newtheorem{thm}{\protect\theoremname}[section]
\theoremstyle{plain}
\newtheorem{lem}[thm]{\protect\lemmaname}
\theoremstyle{definition}
\newtheorem{defn}[thm]{\protect\definitionname}
\theoremstyle{remark}
\newtheorem{rem}[thm]{\protect\remarkname}

\usepackage[english]{babel}

\providecommand{\definitionname}{Definition}

\makeatother

\usepackage{babel}
\providecommand{\definitionname}{Definition}
\providecommand{\lemmaname}{Lemma}
\providecommand{\remarkname}{Remark}
\providecommand{\theoremname}{Theorem}

\begin{document}
\title{\textbf{Semialgebraic Solution \\ of Linear Equation with \\ Continuous
Semialgebraic Coefficients}}
\author{Marcello Malagutti}
\maketitle

\section{Introduction}

\pagenumbering{arabic}

C.$\,$Fefferman gave in \cite{key-1}, by means of analysis techniques,
a necessary and sufficient condition for the existence of a continuous
solution $(\phi_{1},\cdots,\phi_{s})$ of the system 
\begin{equation}
\phi=\sum\limits _{i=1}^{s}\phi_{i}f_{i}\label{eq:System_1}
\end{equation}
given the continuous functions $\phi$ and $f_{i}$. More precisely,
C. Fefferman, applying the theory of the Glaeser refinements for bundles,
proved that system (\ref{eq:System_1}) has a continuous solution
iff the affine Glaeser-stable bundle associated with system (\ref{eq:System_1})
has no empty fiber.

Moreover J.$\,$Kollár, in the same (joint) paper \cite{key-1}, starting
from the above result and making use of algebraic geometry techniques
as blowing up and at singular points, proved that fixed the polynomials
$f_{1},...,f_{s}$ and assuming system (\ref{eq:System_1}) solvable,
then:

\noindent 1) if $\phi$ is semialgebraic then there is a solution
$(\psi_{1},\cdots,\psi_{s})$ of $\phi=\underset{i}{\sum}\psi_{i}f_{i}$
such that the $\psi_{i}$ are also semialgebraic;

\noindent 2) let $U\subset\mathbb{R}^{n}\backslash Z$ (where $Z:=(f_{1}=\cdots=f_{r}=0)$)
be an open set such that $\phi$ is $C^{m}$ on $U$ for some $1\leq m\leq\infty$
or $m=\omega$. Then there is a solution $\psi=(\psi_{1},\cdots,\psi_{s})$
of $\phi=\sum\limits_{i=1}^{s}\psi_{i}f_{i}$ such that the $\psi_{i}$
are also $C^{m}$ on $U$.\medskip{}

In this paper we generalise and prove, by using Fefferman's techniques,
a part of the above results shown by Kollár. More in details, we consider
a \emph{compact metric space} $Q\subseteq\mathbb{R}^{n}$ and a \emph{system
of linear equations}
\begin{equation}
A\left(x\right)\phi\left(x\right)=\gamma\left(x\right),\quad x=(x_{1},\ldots,x_{n})\in Q\label{eq:TargetProblem}
\end{equation}
where
\[
Q\ni x\longmapsto A(x)=(a_{ij}(x))\in M_{r,s}(\mathbb{R})
\]
is \emph{continuous semialgebraic,} with $M_{r,s}(\mathbb{R})$ denoting
the set of real $r\times s$ matrices and 
\[
Q\ni x\longmapsto\gamma(x)\in\mathbb{R}^{r},\gamma(x)=\left[\begin{array}{c}
\gamma_{1}(x)\\
\vdots\\
\gamma_{r}(x)
\end{array}\right]\in\mathbb{R}^{r}
\]
being themselves \emph{continuous semialgebraic functions} on $Q\subseteq\mathbb{R}^{n}$.

\vspace{5mm}

\noindent{\fboxsep 4pt\fbox{\begin{minipage}[t]{1\columnwidth - 2\fboxsep - 2\fboxrule}%
Our aim is to find a \emph{necessary and sufficient condition} for
the existence of a solution $Q\ni x\longmapsto\phi\left(x\right)=\left[\begin{array}{c}
\phi_{1}(x)\\
\vdots\\
\phi_{s}(x)
\end{array}\right]\in\mathbb{R}^{s}$ of system (\ref{eq:TargetProblem}), with the $\phi_{i}:Q\rightarrow\mathbb{R}$
\emph{continuous }and \emph{semialgebraic}.%
\end{minipage}}}

\vspace{0.5cm}

\noindent In particular, we find that a \emph{continuous and semialgebraic
solution} exists if and only if there exits \emph{a continuous solution
and a semialgebraic one (they may possibly be different)} under the
hypothesis that 
\[
B(x,r_{v_{x}})\ni y\longmapsto\gamma_{v_{x}}(y)=\varPi_{H_{y}^{(0)}}v_{x}
\]
is \emph{discontinuous at most at isolated points for each} $x\in Q$
where:

$-$ $H_{y}^{(0)}$ is the fiber at $y\in Q$ of the singular affine
bundle associated to

system (\ref{eq:TargetProblem});

$-$ $B(x,r_{v_{x}})\subset Q$ is an euclidean ball of small radius;

$-$ $v_{x}$ a vector of $H_{x}^{(0)}$;

$-$ $\varPi_{H_{y}^{(0)}}v_{x}$ the projection of $v_{x}$ on $H_{y}^{(0)}$.

\noindent Moreover, we show how to calculate a continuous and semialgebraic
solution of system (\ref{eq:TargetProblem}).

\section{The setting}

Let us start by setting some notations and by showing some important
preliminary results that will be used to pursue our goal. We shall
endow every $\mathbb{R}^{s}$ used here with euclidean norm.
\begin{quote}
\noindent \textbf{Notation 1:} Let $V\subseteq\mathbb{R}^{s}$ be
an affine space in $\mathbb{R}^{s}$ and $w\in\mathbb{R}^{s}$. We
denote the \emph{projection of $w$ on $V$} (i.e. the point $v\in V$
that makes the euclidean norm of $v-w$ as small as possible) by $\Pi_{V}w$.

\medskip{}
\noindent \textbf{Notation 2:} If $x$$\in Q$, consider $Q\ni x\longmapsto A(x)\in M_{r,s}(\mathbb{R})$
and $w\in\mathbb{R}^{s}$. We denote 
\[
\varPi_{1}(x)\,w=\varPi_{\mathrm{\mathrm{Ker}}A(x)^{\perp}}w,\quad\varPi_{2}(x)\,w=\varPi_{\mathrm{Ker}A(x)}w.
\]

\noindent \textbf{Notation 3:} We denote the\emph{ $i$-th column}
of a matrix $A$ by $A_{i}$.
\end{quote}
\begin{lem}
\label{lem:Lemma1}Consider

i) $Q\subseteq\mathbb{R}^{n}$ a compact space,

ii) $Q\ni x\longmapsto A(x)\in M_{r,s}(\mathbb{R})$ a\emph{ }matrix
valued semialgebraic function,

iii) $\phi:Q\rightarrow\mathbb{R}^{s}$ a map,

\noindent  and let $Q\ni x\longmapsto p(x)=\varPi_{1}(x)\,\phi(x)$
be the projection of $\phi$ on $\mathrm{Ker}A(x)^{\perp}$.

\noindent  If $Q\ni x\longmapsto\phi(x)$ is a semialgebraic function
on $Q$ then $Q\ni x\longmapsto p(x)$ is a semialgebraic function
on $Q$.
\end{lem}

\begin{proof}
First of all notice that 
\[
\mathrm{Ker}A(x)^{\perp}=\mathrm{Span}\{A(x)_{i}^{T}\}_{i\in\{1,\ldots,r\}}=\mathrm{Span}\left\{ \text{rows of }A(x)\right\} .
\]

\noindent  Then consider for a given $I\subseteq\left\{ 1,\ldots,r\right\} $
\[
K(I)=\{x\in Q:(A(x)_{\,i}^{T})_{i\in I}\text{ is a basis of }\mathrm{Im}A(x)^{T}\}.
\]
The idea we want to pursue is to project the solution $Q\ni x\longmapsto\phi(x)$
on $\mathrm{Ker}A(x)^{\perp}$ and apply the Tarski-Seidenberg theorem,
thus concluding that the projection is semialgebraic. Actually, to
apply the Tarski-Seidenberg theorem we need that the dimension of
the projection space be independent of $x\in Q$, so we will be localizing
the problem on the $K(I)$. As a matter of fact on $K(I)$ the dimension
of $\mathrm{Ker}A(x)^{\perp}=\mathrm{Span}\{A(x)_{i}^{T}\}_{i\in\{1,\ldots,r\}}=\mathrm{Im}A(x)^{T}$
is constant by definition of $K(I)$. From the semialgebraicity of
$Q\ni x\longmapsto A(x)$ it is then trivial to deduce that $K(I)$
is a semialgebraic (possibly empty) subset on $Q$ since:

\noindent $\bullet\;$the linear independence of $(A(x)_{\,i}^{T})_{i\in I}$
can be translated in terms of the minors of $A(x)$;

\noindent $\bullet\;$the condition that the vectors of $(A(x)_{\,i}^{T})_{i\in I}$
are generators of $\mathrm{Im}A(x)^{T}$ can be expressed by the following
first-order formula
\[
\forall v\in\mathrm{Im}A(x)^{T},\,\exists(\lambda_{\,i}^{T})_{i\in I}\text{ real numbers s.t.}\:v=\underset{i}{\sum}\lambda_{i}A(x)_{\,i}^{T}.
\]
As $\mathrm{Im}A(x)^{T}$ is semialgebraic by definition we get that
that condition is semialgebraic

This is the reason why it suffices to prove the lemma for $K(I)$
since if a function is semialgebraic on a finite collection of semialgebraic
sets then it is semialgebraic on their union:

\noindent $\bullet\;$ if $K(I)=\textrm{Ø}$ there is nothing to
prove;

\noindent $\bullet\;$ if $K(I)\neq\textrm{Ø}$ let
\begin{align*}
\varGamma & (I)=\{(x,\phi(x)):\quad x\in K(I)\}.
\end{align*}

$\circ\;$ If $s\geq r$ we define 
\[
V=\left\{ (x,y):\quad y=A(x)^{T}\begin{pmatrix}\lambda_{1}\\
\vdots\\
\lambda_{r}
\end{pmatrix},\;x\in K(I),\:\lambda_{i}\in\mathbb{R}\right\} \subseteq\mathbb{R}^{n}\times\mathbb{R}^{s}
\]

$\circ\;$ If $s<r$ we define 
\[
V=\left\{ (x,y):\quad y=\widetilde{A}(x)^{T}\begin{pmatrix}\lambda_{1}\\
\vdots\\
\lambda_{\left|I\right|}
\end{pmatrix},\;x\in K(I),\:\lambda_{i}\in\mathbb{R}\right\} \subseteq\mathbb{R}^{n}\times\mathbb{R}^{s}
\]
where $\widetilde{A}(x)\in M_{\left|I\right|,s}(\mathbb{R})$ and
the columns of $\widetilde{A}(x)^{T}$ form a basis of $\mathrm{Im}A(x)^{T}$
given by $(A(x)_{i}^{T})_{i\in I}$. In this case we will still write
$A(x)$ in place of $\widetilde{A}(x)$ and we will still write $r$
for $\left|I\right|$.

\noindent Let us now factorize $A(x)^{T}$ by QR decomposition: $A(x)^{T}=Q(x)R(x)$.\textcolor{black}{{}
As a matter of fact if $m\geq n$ every $m\times n$ matrix can be
written as the product of a squared orthogonal $m\times m$ matrix
$Q$ and a rectangular $m\times n$ matrix $R$ with the blockwise
structure}

\[ R=\left(\begin{array}{c} 
R_{1}\\ 
\hline
0 
\end{array}
\right) \]where $R_{1}$ is an $n\times n$ upper triangular matrix and $0$
is the $\ensuremath{(m-n)\times n}$ zero matrix.

\noindent The set $\Gamma(I)$ is semialgebraic since it is the graph
of a semialgebraic function, hence
\[
\Gamma'(I)=\left\{ \begin{pmatrix}I_{n} & 0\\
0 & Q(x)^{T}
\end{pmatrix}\begin{pmatrix}x\\
y
\end{pmatrix}:\quad(x,y)\in\Gamma(I)\right\} 
\]
 is semialgebraic too. In fact, $Q(x)$ has semialgebraic entries
since the ones of $A(x)$ are semialgebraic and the QR decomposition
can be computed by multiplying iteratively $A(x)$ by appropriate
Householder matrices. These matrices are constructed in the following
way.

\noindent Letting $\left\Vert \cdot\right\Vert $ be the euclidean
norm, $z=\left(A(x)^{T}\right)_{1}$, $v=z+\left\Vert z\right\Vert e_{1}$,
$\alpha=\mathbin{\frac{\left\Vert v\right\Vert ^{2}}{2}}$ and $U_{1}=I_{s}-\frac{vv^{T}}{\alpha}$
where $I_{s}$ is the $s\times s$ identity matrix we have that $U_{1}z=-\left\Vert z\right\Vert e_{1}$.
Now, repeating the procedure on the minor of $A(x)^{T}$ obtained
by eliminating the first row and column (the new $U_{2}$ matrix has
the form $\left(\begin{array}{cc}
I_{1} & 0\\
0 & U'
\end{array}\right)$ where $U'$ is calculated as $U_{1}$\emph{ mutatis mutandis) }we
reach the goal. The $Q$ matrix of the QR decomposition is then the
transpose of the product of all the $U_{i}$ constructed in the previous
way. The result is semialgebraic because in the construction we used
only sums, products and square root extractions of semialgebraic functions
(as the entries of $A(x)$ are semialgebraic) that are semialgebraic
by definition of semialgebraic function.

We next put
\[
V'=\left\{ \begin{pmatrix}I_{n} & 0\\
0 & Q(x)^{T}
\end{pmatrix}v:\quad v\in V\right\} 
\]
and see that if $z\in V'$ then, by the definition of $V'$, we can
write
\[
z=(z_{1},\ldots,z_{n+r},\underbrace{0,\ldots,0}_{s-r})^{T}.
\]
It is important to observe that $V'$ is obviously a vector space
of $\mathbb{R}^{n+s}$ .

Now, recalling the Tarski-Seidenberg theorem\footnote{\textbf{Tarski-Seidenberg Theorem} $Let$ $A$ \emph{a semialgebraic
subset of} $\mathbb{R}^{n+1}$ \emph{and} $\pi:\mathbb{R}^{n+1}\rightarrow\mathbb{R}^{n}$,
\emph{the projection on the first} $n$ \emph{coordinates. Then} $\pi(A)$
\emph{is a semialgebraic subset of} $R^{n}$.

\vspace{2pt}

\textbf{Corollary}\emph{ If} $A$ \emph{is a semialgebraic subset
of} $\mathbb{R}^{n+k}$, \emph{its image by the projection on the
space of the first $n$ coordinates is a semialgebraic subset of}
$\mathbb{R}^{n}$.}, we notice that the projection $\widetilde{\Gamma}'(I)$ of $\Gamma'(I)$
onto $V'$ is semialgebraic. It follows that 

\[
\widehat{\Gamma}'(I)=\left\{ \begin{pmatrix}I_{n} & 0\\
0 & Q(x)
\end{pmatrix}v,\quad v\in\widetilde{\Gamma}'(I)\right\} 
\]
is semialgebraic too. By construction, $\widehat{\Gamma}'(I)$ is
the graph of $\varPi_{1}(x)\,\phi(x),x\in K(I)$ and so the proof
of Lemma \ref{lem:Lemma1} is complete. 
\end{proof}
\begin{thm}
\label{thm:Teorema1}Consider a compact metric space\textup{ $Q\subseteq\mathbb{R}^{n}$}
and the system \emph{(\ref{eq:TargetProblem})}:
\[
A\left(x\right)\phi\left(x\right)=\gamma\left(x\right),\quad x\in Q,
\]
for semialgebraic continuous $A$ and $\gamma$.

\noindent  The system has a semialgebraic solution $\phi_{0}:Q\rightarrow\mathbb{R}^{s}$
iff, given a solution $\phi_{1}:Q\rightarrow\mathbb{R}^{s}$ of the
system,
\[
p(x)=\varPi_{1}(x)\,\phi_{1}(x)\quad is\:semialgebraic.
\]
\end{thm}

\begin{proof}
First of all we show that given a solution $\phi_{1}:Q\rightarrow\mathbb{R}^{s}$
of the system, if $Q\ni x\longmapsto p(x)=\varPi_{1}(x)\,\phi_{1}(x)$
is semialgebraic then there exists a semialgebraic solution of the
system (\ref{eq:TargetProblem}). Notice that 
\[
\phi_{1}(x)=\varPi_{2}(x)\,\phi_{1}(x)+\varPi_{1}(x)\,\phi_{1}(x),\quad\forall x\in Q.
\]
From this we get that
\[
\gamma(x)=A(x)\phi_{1}(x)=A(x)\varPi_{1}(x)\,\phi_{1}(x),\quad\forall x\in Q
\]
 So $Q\ni x\longmapsto p(x)$ is a semialgebraic solution of the system.

\noindent  Conversely we show that, given a solution $\phi_{1}:Q\rightarrow\mathbb{R}^{s}$
of the system, if the system has a semialgebraic solution $\phi_{0}:Q\rightarrow\mathbb{R}^{s}$
then
\[
Q\ni x\longmapsto p(x)=\varPi_{1}(x)\,\phi_{1}(x)\;\text{is semialgebraic.}
\]
In fact
\[
A(x)(\phi_{0}(x)-\phi_{1}(x))=\gamma\left(x\right)-\gamma\left(x\right)=0,\quad\forall x\in Q.
\]
 Decomposing $\phi_{0}(x)$ and $\phi_{1}(x)$ into their components
onto $\mathrm{Ker}A(x)$ and $\mathrm{Ker}A(x)^{\perp}$ we get, respectively,
\begin{align*}
\phi_{0}(x) & =\varPi_{2}(x)\,\phi_{0}(x)+\varPi_{1}(x)\,\phi_{0}(x),\quad\forall x\in Q,\\
\phi_{1}(x) & =\varPi_{2}(x)\,\phi_{1}(x)+\varPi_{1}(x)\,\phi_{1}(x),\quad\forall x\in Q.
\end{align*}
Hence
\[
A(x)(\varPi_{1}(x)\,\phi_{0}(x)-\varPi_{1}(x)\,\phi_{1}(x))=0,\quad\forall x\in Q,
\]
 so that
\[
\varPi_{1}(x)\,\phi_{0}(x)-\varPi_{1}(x)\,\phi_{1}(x)\in\mathrm{Ker}A(x)\bigcap\mathrm{Ker}A(x)^{\perp}=\{0\}.
\]
 Therefore
\[
\varPi_{1}(x)\,\phi_{1}(x)=\varPi_{1}(x)\,\phi_{0}(x),
\]
 and since $Q\ni x\longmapsto\varPi_{1}(x)\,\phi_{0}(x)$ is semialgebraic
by Lemma \ref{lem:Lemma1}, so is $Q\ni x\longmapsto\varPi_{1}(x)\,\phi_{1}(x)$.
\end{proof}
Let us notice that Theorem \ref{thm:Teorema1} shows also that, given
system (\ref{eq:TargetProblem}), $Q\ni x\longmapsto\varPi_{1}(x)\,\phi_{0}(x)$
is uniquely defined on $Q$ since it is independent of the solution
$\phi_{0}$.\vspace{8pt}

At this point let us consider a \emph{singular affine bundle }(or
\emph{bundle} for short) (see \cite{key-1}), meaning a family $\mathcal{H}=(H_{x})_{x\in Q}$
of affine subspaces $H_{x}\subseteq\mathbb{R}^{s}$, parametrized
by the points $x\in Q$. The affine subspaces
\[
H_{x}=\left\{ \lambda\text{\ensuremath{\in}}\mathbb{R}^{s}:A\left(x\right)\lambda=\gamma\left(x\right)\right\} ,\quad x\in Q
\]
 are the \emph{fibers} of the bundle $\mathcal{H}$. (Here, we allow
the empty set $\textrm{Ø}$ and the whole space $\mathbb{R}^{s}$
as affine subspaces of $\mathbb{R}^{s}$.)

Now we define $\mathcal{H}^{(k)}$ to be the $k$-th \emph{Glaeser
refinement} of $\mathcal{H}$ and $\mathcal{H}^{\mathrm{Gl}}$ to
be the \emph{Glaeser-stable refinement }of $\mathcal{H}$ (their fibers
will respectively be denoted by $H_{x}^{(k)}$ and $H_{x}^{\mathrm{Gl}},\quad\forall x\in Q$).
Notice that the projection on the fibers of $\mathcal{H}$ is not
linear as the fibers are affine spaces and not vector spaces.
\begin{lem}
\label{lem:Lemma2}Consider a compact metric space \textup{$(Q,d_{Q}),\:Q\subseteq\mathbb{R}^{n}$}
and an $r\times s$ system of linear equations
\[
A\left(x\right)\phi\left(x\right)=\gamma\left(x\right),\quad x\in Q
\]
where for each $x\in\mathbb{R}^{r}$ the entries of
\[
A(x)=(a_{ij}(x))\in M_{r,s}(\mathbb{R})\quad\text{and }\gamma(x)=(\gamma_{i}(x))\in\mathbb{R}^{r}
\]
are themselves semialgebric functions on $\mathbb{R}^{n}$. 

\noindent  If there is a semialgebraic solution $\phi:Q\rightarrow\mathbb{R}^{s}$
of the system then $\forall x\in Q,\:\forall v_{x}\in H_{x}^{\mathrm{Gl}}$:
\[
Q\ni y\longmapsto\gamma_{v_{x}}(y)\coloneqq\varPi_{H_{y}^{(0)}}v_{x}\text{ is semialgebraic.}
\]
\end{lem}

\begin{proof}
Consider $x\in Q$ and $v_{x}\in H_{x}^{\mathrm{Gl}}$. We define
\[
\gamma_{v_{x}}(y)\coloneqq\varPi_{H_{y}^{(0)}}v_{x},\quad\forall y\in Q.
\]
 By the definition of $H_{y}^{(0)}$ it is true that
\[
H_{y}^{(0)}=\mathrm{Ker}A(y)+\varPi_{1}(y)\,\phi(y),\quad\forall y\in Q,
\]
 and from this expression that
\[
\varPi_{H_{y}^{(0)}}v_{x}=\varPi_{1}(y)\,\phi(y)+\varPi_{2}(y)\,v_{x},\quad\forall y\in Q,
\]
 Now, $Q\ni y\longmapsto\varPi_{2}(y)\,v_{x}=v_{x}-\varPi_{1}(y)\,v_{x}$
is semialgebraic, as $Q\ni y\longmapsto v_{x}$ is semialgebraic (since
it is constant) and so $Q\ni y\longmapsto\varPi_{1}(y)\,v_{x}$ is
semialgebraic by Lemma \ref{lem:Lemma1}. In conclusion, $Q\ni y\longmapsto\varPi_{H_{y}^{(0)}}v_{x}$
is semialgebraic, for it is the sum of semialgebraic functions (recall
that $Q\ni y\longmapsto\varPi_{1}(y)\,\phi(y)$ is also semialgebraic
by Lemma \ref{lem:Lemma1} since $\phi$ is a semialgebraic solution
of system (\ref{eq:TargetProblem})).
\end{proof}
\noindent  After this let us introduce a new notion.
\begin{defn}
Consider a compact metric space $(Q,d),Q\subseteq\mathbb{R}^{n}$.
Given $\mathcal{H}=(H_{x})_{x\in Q}$ whose Glaeser-stable subbundle
is denoted by $(H_{x}^{Gl})_{x\in Q}$, a \emph{semialgebric Glaeser-stable
bundle} associated with the system (\ref{eq:TargetProblem}) is a
family $\widetilde{\mathcal{H}}^{\mathrm{Gl}}=(\widetilde{H}_{x}^{\mathrm{Gl}})_{x\in Q}$
of affine subspaces $\widetilde{H}_{x}^{\mathrm{Gl}}\subseteq\mathbb{R}^{s}$,
parametrized by the points $x\in Q$ , where the fibers $\widetilde{H}_{x}^{\mathrm{Gl}}$
are given by:
\begin{align*}
\widetilde{H}_{x}^{\mathrm{Gl}} & =\{v\in H_{x}^{\mathrm{Gl}}:\exists r_{v}\in\mathbb{R}^{+}\text{ s.t. }\\
 & \qquad\quad B(x,r_{v})\ni y\longmapsto\gamma_{v}(y)=\varPi_{H_{y}^{(0)}}v\text{ is semialgebraic}\}.
\end{align*}
\end{defn}

It is important to notice that $\widetilde{\mathcal{H}}^{\mathrm{Gl}}$
is indeed a bundle, for $\widetilde{H}_{x}^{\mathrm{Gl}}$ is an affine
space, for all $x\in Q$. As a matter of fact:

\noindent $-\;$ if $\widetilde{H}_{x}^{\mathrm{Gl}}=\textrm{Ø}$
the space is affine as we assume the empty space to be an affine space;

\noindent $-\;$ if $\widetilde{H}_{x}^{\mathrm{Gl}}\neq\textrm{Ø}$,
given $v_{0},v_{1},v_{2}\in\widetilde{H}_{x}^{\mathrm{Gl}}$ and $\lambda\in\mathbb{R}$,
we have that
\begin{equation}
(v_{1}-v_{0})+\lambda(v_{2}-v_{0})+v_{0}\in\widetilde{H}_{x}^{\mathrm{Gl}}.\label{eq:Def2.4}
\end{equation}
Property (\ref{eq:Def2.4}) holds because $H_{x}^{\mathrm{Gl}}$ is
an affine space and since $v_{0},v_{1},v_{2}\in\widetilde{H}_{x}^{\mathrm{Gl}}\subseteq H_{x}^{\mathrm{Gl}}$
we have

\[
(v_{1}-v_{0})+\lambda(v_{2}-v_{0})+v_{0}\in H_{x}^{\mathrm{Gl}}.
\]
Moreover
\[
\forall i\in\{0,1,2\},\,\exists r_{v_{i}}\in\mathbb{R}^{+}\text{ s.t. }\;B(x,r_{v_{i}})\ni y\longmapsto\gamma_{i}(y)=\varPi_{H_{y}^{(0)}}v_{i}\text{ is semialgebraic}.
\]
Considering $\overline{r}=\min\{r_{v_{0}},r_{v_{1}},r_{v_{2}}\}$
we get that on $B(x,\overline{r})$
\[
\gamma(y):=\varPi_{H_{y}^{(0)}}((v_{1}-v_{0})+\lambda(v_{2}-v_{0})+v_{0}).
\]
We now consider the orthogonal decomposition of $\gamma(y)$ on to
$\mathrm{Ker}A(y)$ and $\mathrm{Ker}A(y)^{\bot}$
\[
\gamma(y)=\varPi_{1}(y)\varPi_{H_{y}^{(0)}}(v_{1}+\lambda(v_{2}-v_{0}))+\varPi_{2}(y)\varPi_{H_{y}^{(0)}}(v_{1}+\lambda(v_{2}-v_{0})).
\]
Recalling then that the projection of a solution of system (\ref{eq:TargetProblem})
on $\mathrm{Ker}A(y)^{\bot}$ is unique and considering $B(x,\overline{r})\ni y\longmapsto p(y)=\varPi_{1}\gamma_{v_{i}}(y),\quad i=0,1,2,$
gives
\[
\gamma(y)=p(y)+\varPi_{2}(y)(v_{1}+\lambda(v_{2}-v_{0}))
\]
Note that $p$ is semialgebraic on $B(x,\overline{r})$ by Lemma \ref{lem:Lemma1}
and by the uniqueness of $p$ (that is $p(y)=\varPi_{1}(y)\varPi_{H_{y}^{(0)}}v_{0}$
and $\varPi_{H_{y}^{(0)}}v_{0}$ is semialgebraic by hypothesis).
Moreover 
\[
B(x,\overline{r})\ni y\longmapsto\varPi_{2}(y)\,(v_{1}+\lambda(v_{2}-v_{0}))=(v_{1}+\lambda(v_{2}-v_{0}))-\varPi_{1}(y)\,(v_{1}+\lambda(v_{2}-v_{0}))
\]
 is semialgebraic by Lemma \ref{lem:Lemma1} as $B(x,\overline{r})\ni y\longmapsto v_{1}+\lambda(v_{2}-v_{0})$
is semialgebraic since it is constant.
\begin{rem}
\label{rem:Oss1}By Lemma \ref{lem:Lemma2} it follows that if a semialgebraic
solution of system (\ref{eq:TargetProblem}) exists then $\widetilde{\mathcal{H}}^{\mathrm{Gl}}=\mathcal{H}^{\mathrm{Gl}}$.
Moreover if for an $x\in Q$ there is a $v_{x}\in H_{x}^{\mathrm{Gl}}$
such that $B(x,r_{v_{x}})\ni y\longmapsto\gamma_{v_{x}}(y)=\varPi_{H_{y}^{(0)}}v_{x}$
is semialgebraic then $\widetilde{H}_{x}^{\mathrm{Gl}}=H_{x}^{\mathrm{Gl}}$
since if $v'\in H_{x}^{\mathrm{Gl}}$ then
\[
\begin{array}{ccc}
B(x,r_{v_{x}})\ni y\longmapsto\varPi_{H_{y}^{(0)}}v' & = & \varPi_{1}(y)\,\varPi_{H_{y}^{(0)}}v'+\varPi_{2}(y)\,\varPi_{H_{y}^{(0)}}v'\\
 & = & \varPi_{1}(y)\,\gamma_{v_{x}}(y)+\varPi_{2}(y)\,v'\\
 & = & \varPi_{1}(y)\,\gamma_{v_{x}}(y)+v'-\varPi_{1}(y)\,v'
\end{array}
\]
 where the equality in the second line is due to the uniqueness of
the projection of a solution and to the fact that $H_{y}^{(0)}$ is
parallel to $\mathrm{\mathrm{Ker}}A(y)$.

Therefore $B(x,r_{v_{x}})\ni y\longmapsto\varPi_{H_{y}^{(0)}}v'=\varPi_{1}(y)\,\varPi_{H_{y}^{(0)}}v'+\varPi_{2}(y)\,\varPi_{H_{y}^{(0)}}v'$
is semialgebraic as it is a sum of semialgebraic functions by Lemma
\ref{lem:Lemma1} and thus $v_{x}\in\widetilde{H}_{x}^{\mathrm{Gl}}$.
\end{rem}

\section{Existence of a continuous semialgebraic solution}
\begin{thm}
\label{thm:LastThm}Consider a compact metric space\textup{ $Q\subseteq\mathbb{R}^{n}$}
and a system of linear equations
\begin{equation}
A\left(x\right)\phi\left(x\right)=\gamma\left(x\right),\quad x\in Q\label{eq:TargetProblemSystem}
\end{equation}
where the entries of
\[
A(x)=(a_{ij}(x_{1},\ldots,x_{n}))\in M_{r,s}(\mathbb{R})\quad\text{and}\quad\gamma(x)=\left(\gamma_{i}(x)\right)\in\mathbb{R}^{r}
\]
are themselves semialgebric functions on $\mathbb{R}^{n}$.

\noindent  Assume that for every given $x\in Q$ such that $\widetilde{H}_{x}^{\mathrm{Gl}}\neq\emptyset$
there exists $v_{x}\in\widetilde{H}_{x}^{\mathrm{Gl}}$ such that
$B(x,r_{v_{x}})\ni y\longmapsto\gamma_{v_{x}}(y)=\varPi_{H_{y}^{(0)}}v_{x}$
is discontinuous, at most, at isolated points (that therefore must
be finitely many). Then system \emph{(\ref{eq:TargetProblemSystem})}
has a continuous semialgebraic solution $\phi:Q\rightarrow\mathbb{R}^{s}$
iff $\widetilde{\mathcal{H}}^{\mathrm{Gl}}$ has no empty fiber.
\end{thm}

\begin{proof}
\ 

\noindent  $\bullet\;$ At first we show that if $\widetilde{H}_{x}^{\mathrm{Gl}}$
has no empty fiber then the system (\ref{eq:TargetProblem}) has a
continuous semialgebraic solution $\phi:Q\rightarrow\mathbb{R}^{s}$.

\noindent  By hypothesis we have that $\forall x\in Q,\:\exists v_{x}\in H_{x}^{\mathrm{Gl}}\neq\textrm{Ø, so that }\exists r_{v_{x}}\in\mathbb{R}^{+}$
s.t. $B(x,\overline{r}_{v_{x}})\ni y\longmapsto\gamma_{v_{x}}(y)=\varPi_{H_{y}^{(0)}}v_{x}$
is semialgebraic and possibly discontinuous at isolated points.

\noindent  We next claim that
\begin{align}
\forall x & \in Q,\:\forall v_{x}\in\widetilde{H}_{x}^{\mathrm{Gl}}\neq\textrm{Ø},\exists\overline{r}_{v_{x}}\text{ with }0<\overline{r}_{v_{x}}<r_{v_{x}}\:\text{ s.t.}\label{eq:IpotesiAssurda1}\\
 & \quad B(x,\overline{r}_{v_{x}})\ni y\longmapsto\gamma_{v_{x}}(y)=\varPi_{H_{y}^{(0)}}v_{x}\text{ is continuous}.\nonumber 
\end{align}
 To prove the claim, let us, by contradiction, assume that
\begin{align*}
\exists x & \in Q,\:\exists v_{x}\in\widetilde{H}_{x}^{\mathrm{Gl}}:\forall n\in\mathbb{N},\,\exists y_{n}\in B(x,\frac{r_{v_{x}}}{n})\text{ s.t. }\\
 & \quad B(x,r_{v_{x}})\ni y\longmapsto\gamma_{v_{x}}(y)\text{ is discontinuous at }y_{n}.
\end{align*}
 We have two cases:

1. $\forall n\in\mathbb{N}\text{ one has }y_{n}\neq x$. Then $\gamma_{v_{x}}(y)$
has infinitely many isolated discontinuity points which is impossible
because $\gamma_{v_{x}}$is semialgebraic;

2. $\exists\overline{n}\in\mathbb{N}\text{ such that }y_{\overline{n}}=x$.
If $v_{x}\notin H_{x}^{\mathrm{Gl}}$ then $v_{x}\notin\widetilde{H}_{x}^{\mathrm{Gl}}$
since $H_{x}^{\mathrm{Gl}}\supseteq\widetilde{H}_{x}^{\mathrm{Gl}}$:
this is clearly not possible as we took $v_{x}\in\widetilde{H}_{x}^{\mathrm{Gl}}$.
Now, if we had $v_{x}\in H_{x}^{\mathrm{Gl}}$ then on the one hand
\[
\mathrm{dist}(v_{x};H_{y}^{\mathrm{Gl}})\underset{{\scriptstyle y\rightarrow x}}{\longrightarrow}0,
\]
and, on the other, 
\[
\left\Vert \Pi_{H_{y}^{(0)}}v_{x}-v_{x}\right\Vert =dist(v_{x},H_{y}^{(0)})\underset{{\scriptstyle H_{y}^{\mathrm{Gl}}\subseteq H_{y}^{(0)}}}{\underbrace{\leq}}d(v_{x},H_{y}^{\mathrm{Gl}}).
\]
 Therefore
\[
\Pi_{H_{y}^{(0)}}v_{x}\underset{{\scriptstyle y\rightarrow x}}{\longrightarrow}v_{x}\underset{\overset{\uparrow}{{\scriptstyle v_{x}\in H_{x}^{\mathrm{Gl}}\subseteq H_{x}^{(0)}}}}{=}\Pi_{H_{x}^{(0)}}v_{x}.
\]
This is impossible since $\gamma_{v_{x}}(y)$ would be continuous
at $x$, contrary to the assumption. Thus (\ref{eq:IpotesiAssurda1})
holds.

\vspace{15pt}

\noindent  Now notice that the set of balls $\{B(x,\overline{r}_{v_{x}})\}_{x\in Q}$,
where $v_{x}$ is chosen in $\widetilde{H}_{x}^{\mathrm{Gl}}$, is
an open cover of the compact space $Q$. Then there is $N$ such that
$\{B(x_{i},\overline{r}_{v_{x_{i}}})\}_{i=1,\ldots,N}$ is an open
cover of $Q$. Consider
\[
\tau_{(x,r)}(y)=\begin{cases}
\sqrt{r^{2}-\left\Vert y-x\right\Vert ^{2}} & \text{for }y\in B(x,r),\\
0 & \text{for }y\notin B(x,r).
\end{cases}
\]
 Notice that $\tau_{(x,r)}(y)$ is semialgebraic and continuous on
$Q$, $\forall x\in Q$, $\forall r\in\mathbb{R}^{+}$ and that ${\displaystyle \sum_{i=1}^{N}}\tau_{(x_{i},\overline{r}_{v_{x_{i}}})}(y)>0$
for each $y\in Q$ as $\tau_{(x,r)}(y)\geq0$ for every $y\in Q$
and $\tau_{(x,r)}(y)>0$ for all $y\in B(x,r)$. Moreover, $\forall y\in Q,\,\exists B(x_{i},\overline{r}_{v_{x_{i}}})\text{ as above s.t. }y\in B(x_{i},\overline{r}_{v_{x_{i}}})$
since $\{B(x_{i},\overline{r}_{v_{x_{i}}})\}_{i=1,\ldots,N}$ is an
open covering of $Q$. Hence the function 
\[
\phi(y)=\frac{1}{{\displaystyle \sum_{i=1}^{N}}\tau_{(x_{i},\overline{r}_{v_{x_{i}}})}(y)}\sum_{j=1}^{N}\tau_{(x_{j},\overline{r}_{v_{x_{j}}})}(y)\Pi_{H_{y}^{(0)}}v_{x_{j}}
\]
 is a semialgebraic and continuous solution of the system on $Q$.

\vspace{15pt}

$\bullet\;$ Conversely, we show that if system (\ref{eq:TargetProblem})
has a continuous semialgebraic solution $\phi:Q\rightarrow\mathbb{R}^{s}$
then $\widetilde{H}_{x}^{\mathrm{Gl}}$ has no empty fiber.

\noindent By Remark \ref{rem:Oss1} we have that $\widetilde{\mathcal{H}}^{\mathrm{Gl}}=\mathcal{H}^{\mathrm{Gl}}$
but $\mathcal{H}^{\mathrm{Gl}}$ has no empty fiber because there
is a continuous solution of the system (\ref{eq:TargetProblem}) as
shown in \cite{key-1}.

\vspace{15pt}

\noindent The proof of the theorem is complete.
\end{proof}
Theorem \ref{thm:LastThm} gives therefore an answer to the initial
problem of determining a \emph{necessary and sufficient condition}
for the existence of a solution $Q\ni x\longmapsto\phi\left(x\right)=\left[\begin{array}{c}
\phi_{1}(x)\\
\vdots\\
\phi_{s}(x)
\end{array}\right]\in\mathbb{R}^{s}$ of the system (\ref{eq:TargetProblem}), with the $\phi_{i}:Q\rightarrow\mathbb{R}$
\emph{continuous }and \emph{semialgebric}.
\begin{rem}
\label{rem:Projection}Let $p$ be the projection of a solution of
the system (\ref{eq:TargetProblem}) on $\mathrm{Ker}A(y)^{\perp},\,y\in Q$.
If $p$ is not semialgebraic the system has no semialgebraic solution
by Theorem \ref{thm:Teorema1} and so it has no continuous and semialgebraic
solution. Otherwise, if $p$ is semialgebraic then $\widetilde{\mathcal{H}}^{\mathrm{Gl}}=\mathcal{H}^{\mathrm{Gl}}$
since if $v\in H_{x}^{\mathrm{Gl}}$ we may write
\[
Q\ni y\longmapsto\varPi_{H_{y}^{(0)}}v=p(y)+\varPi_{2}(y)\,v=p(y)+v-\varPi_{1}(y)\,v
\]
 that is semialgebraic by Lemma \ref{lem:Lemma1}. We have that $\gamma_{v_{x}}(y)=p(y)+\varPi_{1}(y)\,v_{x}$,
$\forall y\in B(x,r_{v_{x}})$ since $\gamma_{v_{x}}(y)$ is the projection
of $v_{x}$ on $H_{y}^{(0)}$. This implies that for each $x\in Q$
there is a $v_{x}$ such that $\gamma_{v_{x}}$ is semialgebraic iff
$p$ is semialgebraic and $\gamma_{v_{x}}$ is discontinuous at most
at isolated points for each $x\in Q$ iff $p$ and $B(x,r_{v_{x}})\ni y\longmapsto\varPi_{1}(y)\,v_{x}$
are discontinuous at most at isolated points for each $x\in Q$. Note
that if there is a $v_{x}$ such that $\gamma_{v_{x}}$ is semialgebraic
then $\gamma_{v_{x}}$ is semialgebraic for all $v\in H^{\mathrm{Gl}}$
by Remark \ref{rem:Oss1}.

Hence, if we know a solution $\phi$ of system (\ref{eq:TargetProblem})
or, at least, its projection on $\mathrm{Ker}A(x)^{\perp}$, Theorem
\ref{thm:LastThm} can be written in the following (equivalent) form.
\end{rem}

\begin{thm}
Let $p$ be the projection of a solution of the system on $\mathrm{Ker}A(x)^{\perp}$
and assume that for every given $x\in Q$ there exists $v_{x}\in\widetilde{H}_{x}^{\mathrm{Gl}}$
such that $B(x,r_{v_{x}})\ni y\longmapsto\gamma_{v_{x}}(y)=\varPi_{H_{y}^{(0)}}v_{x}$
is discontinuous, at most, at isolated points (which is the same as
saying that $p$ and $B(x,r_{v_{x}})\ni y\longmapsto\varPi_{1}(y)\,v_{x}$
are discontinuous at most at isolated points for every given $x\in Q$).
We have that the following conditions are equivalent:

i) The system has a continuous and semialgebraic solution.

ii) $\widetilde{\mathcal{H}}^{\mathrm{Gl}}$ has no empty fiber.

iii) The Glaeser-stable refinement $\mathcal{H}$ associated with
the system has no empty fibers and $p$ is semialgebraic.

iv) The system has a continuous solution and a semialgebraic one (they
may possibly be different).
\end{thm}

\begin{proof}
\

We showed that $i)\Leftrightarrow ii)$ in Theorem \ref{thm:LastThm}
(using also Remark \ref{rem:Projection}). The proof of $ii)\Leftrightarrow iii)$
follows from Remark \ref{rem:Projection} and that of $iii)\Longleftrightarrow iv)$
from  Fefferman's result that a system like (\ref{eq:TargetProblem})
has a continuous solution iff the Glaeser-stable refinement $\mathcal{H}$
associated with the system has no empty fibers (see \cite{key-1})
and from Theorem \ref{thm:Teorema1}.
\end{proof}

\end{document}